\documentclass{amsart}
\usepackage{amssymb}
\usepackage{latexsym}
\usepackage{amsmath}
\usepackage{enumerate}
\usepackage{amsthm}
\usepackage{pifont}
\usepackage{tikz}
\usepackage[f]{esvect}
 \usepackage{dutchcal}
  \usepackage{hyperref}

\newcommand{\qee} {\hspace*{2mm}\hfill \ding{109}}

\renewcommand{\iff}{\leftrightarrow}
\renewcommand{\leq}{\leqslant}
\renewcommand{\geq}{\geqslant}
\renewcommand{\preceq}{\preccurlyeq}

\renewcommand{\phi}{\varphi}

\renewcommand{\Theta}{\varTheta}
\renewcommand{\Phi}{\varPhi}
\renewcommand{\Psi}{\varPsi}
\renewcommand{\Xi}{\varXi}
\renewcommand{\Omega}{\varOmega}
\renewcommand{\Gamma}{\varGamma}

\newtheorem{theorem}{Theorem}[section]
\newtheorem{define}[theorem]{Definition}

\newtheorem{exa}[theorem]{Example}

\newtheorem{exerc}[theorem]{Exercise}

\newtheorem{conj}[theorem]{Conjecture}

\newtheorem{ques}[theorem]{Open Question}

\newtheorem{lem}[theorem]{Lemma}

\newtheorem{cor}[theorem]{Corollary}

\newtheorem{rem}[theorem]{Remark}

\DeclareMathOperator{\necessary}{\text{\tikz[scale=.6ex/1cm,baseline=-.6ex,line width=.1ex]{
                            \draw (-1,-1) rectangle (1,1);}}}

\newcommand{\mc}[1]{\mathcal {#1}}
\newcommand{\mf}[1]{{\mathfrak {#1}}}

\newcommand{\To}{\Rightarrow}

\newcommand{\bleq}{\mathbin{\leq}}
\newcommand{\bles}{\mathbin{<}}
\newcommand{\gnum}[1]{{\ulcorner #1 \urcorner}}
\newcommand{\opr}{\necessary}

\newcommand{\sspl}[4]{\big(\begin{smallmatrix}
  #1 & #2\\
  #3 & #4
\end{smallmatrix}\big)}

\usepackage{upgreek}

\newcommand{\bs}{{\sf BeSh}}
\newcommand{\jer}{Je{\v{r}}{\'{a}}bek}
\newcommand{\grullet}{\text{\textcolor{gray}{$\bullet$}}}

\title{On a Theorem by Bezboruah and Shepherdson}

\author{Albert Visser}
 \address{Department of Philosophy and Religious Studies,
                Utrecht University,
               Janskerkhof 13,
               3512 BL~~Utrecht, The Netherlands}
\email{a.visser@uu.nl}



\thanks{I am grateful to Emil \jer\ for careful comments on an earlier version of the paper.
I also thank him for his gracious permission to reproduce his argument concerning the cut
introduced in his paper \cite{jera:sequ12}. See our Appendix~\ref{emilsmurf}.
I am grateful to Mateusz {\L}e{\l}yk for asking a good question. 
I thank  Richard Heck, Joost Joosten and Leon Probst who suggested improvements to the first version of the preprint.
I worked with ChatGPT to develop the proof of Lemma~\ref{goldensmurf}. 
See
Remark~\ref{chatsmurf}.}

\begin{document}

\begin{abstract}
We discuss the incompleteness result proven by Bezboruah and Shepherdson in  \cite{bezb:gode76}.
This result tells us that the weak theory ${\sf PA}^-$ does not prove the consistency of any theory
(under certain assumptions explained in the paper). Kreisel argued that such a result is not meaningful.
We discuss Kreisel's objection and conclude that his argument does not hold water. We compare 
Pudl\'ak's extension of the Second Incompleteness Theorem with the  Bezboruah-Sheperdson Theorem.
Finally, we reprove the Bezboruah-Sheperdson Theorem for a sequence-coding based on an insight of
Nielsen and Markov.
\end{abstract}

\maketitle

\section{Introduction}
Suppose we are thinking about a Great Theorem. A first question might be: what is the scope of this Great Theorem? 
Suppose we prove a case that is somewhat removed
from the originally intended range of applicability of the Great Theorem. One of the vocal leaders of the field argues that our solution for 
such an outlier case cannot possibly be significant. Suppose also our proof is quite different from the original proof. Is the leader right and should
we abort our project? Did we prove the same theorem at all? What to say and what to
do?
\dots   

This is the scenario faced by Amala Bezboruah and John C. Shepherdson around 1976. The Great Theorem is the Second Incompleteness Theorem
and the argumentative leader is Georg Kreisel. Bezboruah and Shepherdson proved that a weak theory, to wit ${\sf PA}^-$, does not prove the consistency
of a very weak theory (in fact a sub-theory of propostional logic). See Theorem~\ref{bezshe} of the present paper. Their result implies that {\sf Q} does not prove its own
consistency.
Bezboruah and Shepherdson choose to publish their result in the paper  \cite{bezb:gode76} in spite of Kreisel's
objections,
but the tone of the introduction is somewhat apologetic, accepting Kreisel's argument.
See our Section~\ref{kreiselsmurf} which quotes the apologetic remark.

Afterwards, the paper has, in my view, been unjustly neglected.
Google Scholar counts 63 citations, which is not many for a paper that came out in 1976. 
Of course, it is hard to judge whether this is, in part, due to the long shadow cast by Kreisel, but it is very well possible.
Kreisel's negative opinion is certainly echoed in more recent publications and in FOM discussions. For example,
   Panu Raatikainen writes in the Stanford Encyclopedia of Philosophy: 
\begin{quote} {\small
    There is a version of the second incompleteness theorem for {\sf Q} (see Bezboruah {\&} Shepherdson 1976) but there has 
     been some debate on whether the relevant statement in {\sf Q} can really be taken to express consistency, {\sf Q} being so weak [\dots].
     \cite{raat:gode25}}
    \end{quote}

In this paper, I discuss the Bezboruah {\&} Shepherdson paper. In Section~\ref{kreiselsmurf}, I present Kreisel's objection. 
I will argue that Kreisel's argument does not really hold water.
In Section~\ref{gopusmurf}, I compare the Bezboruah-Shepherdson result with a modern version of the Second Incompleteness Theorem
due to Pavel Pudl\'ak. The discussion suggests that  the Bezboruah-Shepherdson result and the Second Incompleteness Theorem in a modern
version can be best viewed as entirely different results with some overlapping special cases.
        In Section~\ref{markovsmurf}, I will give a proof of the Bezboruah-Shepherdson result for a different sequence coding based on
        ideas of Nielsen and Markov.
   
   \section{Kreisel}\label{kreiselsmurf}
   Here is the apologetic passage concerning Kreisel's critical remarks in \cite{bezb:gode76}.
    \begin{quote}
   {\small
   Kreisel has argued [\dots] that something like
the Hilbert-Bernays conditions represent a completely satisfactory statement of
G\"odel's second theorem on the grounds that ${\sf Th}(x)$ can hardly be considered to
express the notion `$x$ is the g.n. of a theorem of $U$' if such simple properties of
proofs are not formally provable. [...]
   Thus for such a weak system as Raphael Robinson's {\sf Q} [\dots] which cannot even
prove commutativity of addition, it is apparently [\dots] unknown whether the
usual formulae expressing its consistency are unprovable. We must agree with
Kreisel that this is devoid of any philosophical interest and that in such a weak
system this formula cannot be said to express consistency but only an algebraic
property which in a stronger system (e.g. Peano arithmetic {\sf P}) could reasonably
be said to express the consistency of {\sf Q}. Nevertheless there is a technical challenge
and we respond to it here by showing that, for at least one natural way of expressing
consistency, the consistency of {\sf Q} cannot be proved in {\sf Q}. \cite[pp 503, 504]{bezb:gode76}
   }
   \end{quote}

So, Kreisel claims that, in the context of a theory like {\sf Q}, a usual sentence ${\sf con}({\sf Q})$ fails to express the
consistency of {\sf Q}, since {\sf Q} does not verify desired properties of the
provability predicate. Apparently, for Kreisel, the same sentence ${\sf con}({\sf Q})$  \emph{can} express the consistency of {\sf Q} in a stronger theory that does verify
the Hilbert-Bernays or the L\"ob Conditions. This last tells us at least that Kreisel's point of view does not entail skepticism concerning the
question whether a numerical statement can somehow express a non-numerical fact. Let us simply postulate that, in the
context of true arithmetic, an appropriate formula ${\sf con}({\sf Q})$ does indeed express the consistency of {\sf Q}. 
Then, why can we not view {\sf Q} as a theory with as intended semantics the natural numbers? Where does the demand that
{\sf Q} considered \emph{all in isolation} has to provide the meaning of the formulas come from? Why should meaning vary with theory? Even if we are inferentialists,
questions about what subsystem of our preferred system is needed to prove so-and-so hold some interest.
There is no need to view inferences in the weaker system as having different meanings and, thus, being different inferences. 
The meanings could still be the meanings provided by the
stronger system.\footnote{Suppose we have proven something, e.g. division by 3, in {\sf ZFC} and we ask the question \emph{Can we also prove this without Choice?}
Then, the answer \emph{No, you can't, since in {\sf ZF} the statement would have a different meaning} seems a bit silly.}
Thus, Kreisel's argument does not seem very convincing.

\begin{rem}{\small
With some fantasy, we can give Kreisel's take on the Second Incompleteness Theorem a more radical reading. 
The claim would be that something like the following statement is the pro-version of the Second Incompleteness Theorem. 
\begin{quotation}
Consider any theory $U$ that interprets {\sf R} via a $k$-dimensional $N$ and any reasonable G\"odel numbering $\gamma$.
Suppose a formula $\phi(\vv x)$ with $k$ free variables
is given and suppose that $U$ proves the L\"ob Conditions for $\phi$ w.r.t. $N,\gamma$. Then,
$U$ proves L\"ob's Principle for $\phi$ w.r.t. $N,\gamma$.
\end{quotation}
We note that we need assume nothing further about the theory $U$ in the formulation. It  need not be computably enumerable, etcetera. Moreover,
the constraints on $\gamma$ are light. We only have to
number sentences, since codes for proofs do not enter the formulation. We do need to be able to verify the diagonal lemma, though.

The predicates $\phi(x)$ that satisfy the L\"ob Conditions encompass many predicates, like $x=x$, that
intuitively have nothing to do with provability. 

My 5 cents would be that we want to keep a statement of the Second Incompleteness Theorem that is directed at a predicate that represents intuitively
a provability predicate. We keep the insight about the step from L\"ob Conditions to L\"ob's Principle as a powerful lemma that can be used in 
versions of the proof of the Second Incompleteness Theorem.

We refer the reader to \cite[p 38]{fefe:arit60} for Solomon Feferman's objections to a similar idea.
}\end{rem}

A different approach to Kreisel's objection could be to grant Kreisel the idea that we need good closure conditions for
provability, but claim that L\"ob's Conditions are too strong. One approach would be to consider the generalised
provability conditions as discussed in \cite[pp 173-174]{haje:meta91}.  A second approach is the one discussed in
Appendix~\ref{loru}. There we show that, in the case of ${\sf PA}^-$, we can verify, if we choose our proof system and sequence coding wisely,  the validity of a very decent modal logic
{\sf WfL},
which is not quite L\"ob's Logic.
In {\sf WfL}, we do have the de Jongh-Sambin-Bernardi Theorem about the uniqueness of modalised fixed points. So, the G\"odel sentence is
still unique. However, in {\sf WfL}, we cannot prove the de Jongh-Sambin Theorem about the explicit definability. For example, we 
 cannot show that the G\"odel fixed point is equivalent to the consistency statement. (This does not mean that ${\sf PA}^-$ does not
 prove this equivalence. However, it is very implausible that it does.) 
So, even if we grant Kreisel's argument, there could still disagreement about where to draw the boundary.

If, as I think, Kreisel's argument does not hold water, a general objection against the consideration
of a question like the validity of the Second Incompleteness Theorem for {\sf Q} disappears.   This, of course, does not
automatically imply that the question is sensible or interesting.  So, can we give some positive arguments for its interest?
Of course, as Bezboruah and Shepherdson say, proving the Second Incompleteness Theorem for {\sf Q} is a wonderful technical challenge.
I felt the same for proving my own variation presented in Section~\ref{markovsmurf}. These proofs deliver constructions of discretely ordered
rings with certain special properties. The model constructed by Bezboruah and Shepherdson is recursive and the reason that
it contains an inconsistency proof is completely perspicuous. This has some didactic value since, in the case of the
inconsistency proofs in existing in models of even weak theories like ${\sf S}^1_2$, we cannot really visualise such proofs.
My version in Section~\ref{markovsmurf} illustrates the difference between purely universal sentences and $\Pi^0_1$-sentences
even of very low depth of quantifier changes. Finally, I do think the contrast between G\"odel's Theorem and Bezboruah {\&} Shepherdson's
Theorem has didactic value. I hope this will be illustrated by the discussion in Section~\ref{gopusmurf}.  

    \begin{rem}{\small
    Here is a speculative remark about the question why Kreisel was such a firm believer in the Hilbert-Bernays conditions as necessary condition for
    intensional correctness. The remark is all the more speculative, since I have not seen a text by Kreisel himself directly addressing this issue.
    
    The story of the discovery of L\"ob's Theorem is well-known. In \cite{henk:prob52}, Leon Henkin asked
  whether the G\"odel Fixed Point of ${\sf prov}_{\sf PA}(x)$ is provable or independent. Kreisel answered with a resounding \emph{it depends} in his paper
  \cite{krei:prob53}. He constructed versions of the provability predicate that lead to different answers. In \cite{loeb:solu55}, Martin L\"ob isolated
   natural conditions for a provability predicate that led to the answer that the Henkin Sentence was provable.\footnote{L\"ob was a very independent
   thinker, which, at least partly, explains why he was not discouraged by Kreisel's paper to look into the matter. 
    L\"ob says that Kreisel did not prove his result for a proof-predicate  $\mf B(x,y)$  expressing that $x$ is the number of a formal proof of the formula with
    G\"odel number $y$.
    But what does this mean?
     The beauty of his paper is that, for the purpose of answering Henkin's question, we do not need to worry about the precise answer. 
     His proof reveals the required conditions on such a predicate
    to produce a definite answer. We only need to recognise that our $\mf B$ satisfies these conditions.}
    The L\"ob Conditions are
   closely related to the Hibert-Bernays Conditions. 
   
   See \cite{halb:henk14} for a discussion of the history and
   the ins and outs of Kreisel's examples. In the light of what he viewed as his methodological blunder, Kreisel met his Damascus and converted to a standpoint diametrically opposite to
   the  philosophy of  \cite{krei:prob53}.
   He proposed to make verifiability  in the ambient theory of conditions like the L\"ob Conditions a prerequisite for
   being a sensible provability predicate at all.\footnote{Another consequence of Kreisel's \emph{Damascus} was that he believed that it could be well
   possible that we find conditions to settle the question whether the Rosser fixed point delivers a unique sentence modulo provable equivalence
    (personal conversation). This
   in spite of the resounding \emph{it depends} pronounced by David Guaspari and Robert Solovay in \cite{guas:ross79} and in opposition to
           the subsequent judgement of most relevant specialists.} 
}   \end{rem}

   \section{Bezboruah {\&} Shepherdson meet G\"odel {\&} Pudl\'ak}\label{gopusmurf}
   We are in a different situation than Bezboruah {\&} Shepherdson when they wrote their article.
   There was no adaptation/extension of G\"odel's orginal proof that could reach as low down as {\sf Q}.
  However, there is one now. The decisive step towards that was taken by Pavel Pudl\'ak. Let us first 
  take a closer look at the Second Incompleteness Theorem. 
  
  In the usual presentation, the theorem says that  a theory that contains a sufficient amount of arithmetic
  does not prove its own consistency, unless it is inconsistent. A reasonable further explication of this statement is as follows.
  Let {\sf B} be a given arithmetical base theory. (\$) If a c.e. theory $U$ extends {\sf B} and $U$ proves ${\sf con}(U)$, then it is inconsistent.
  Here ${\sf con}(U)$ is an arithmetical representation of the consistency statement for $U$. We have briefly touched on the
  question what the constraints on the pair $U$, ${\sf con}(U)$ should (not) be in Section~\ref{kreiselsmurf}. In this section, we
  will pretend that it is sufficiently clear what it is to have a suitable formula ${\sf con}(U)$.  We note that (\$)
  is equivalent to: (\#) if a c.e. theory $U$ extends ${\sf B}+{\sf con}(U)$, then $U$ is inconsistent.
  
   We note that (\#) is still not adequate. G\"odel clearly knew that his result worked also for e.g. set theory,
   a theory in a different signature. This means that the notion of extension employed in (\#) is to
   restrictive. A better idea is to replace extension with \emph{being interpretable in}.
   This gives us the  following version: (\textsection) if a c.e. theory $U$ interprets ${\sf B}+{\sf con}(U)$, then $U$ is inconsistent.
   
   The appropriate theory {\sf B} was not yet available in the early days of the Incompleteness Theorems. My favourite choice was
   provided by Sam Buss in \cite{buss:boun86}, to wit his theory for p-time computability ${\sf S}^1_2$.
   The point here is that, in ${\sf S}^1_2$, we can repeat G\"odel's reasoning in a completely natural way. So, here
   is my preferred version of the Second Incompleteness Theorem.
   
   \begin{theorem}[Kurt G\"odel with some help from Sam Buss]\label{gobu}
   Let $U$ be a c.e. theory. If $U$ interprets ${\sf S}^1_2+{\sf con}(U)$, then $U$ is inconsistent.
   \end{theorem}
   
     \begin{rem}{\small
   We can strengthen Theorem~\ref{gobu} in various other ways. For example, it would also work for forcing interpretations. Also,
   we can prove variants for non-c.e. theories. See e.g. \cite{viss:look19}. 
   }
   \end{rem}

   Of course, a lot more has to be said about what ${\sf con}(U)$ is but that is beyond the scope of the present paper.
   To connect Theorem~\ref{gobu} with {\sf Q}, we need an insight due to Pavel Pudl\'ak. 
   The basic idea is to interpret ${\sf S}^1_2$ on a definable cut, say, $\mc J$. If {\sf Q} would prove ${\sf con}({\sf Q})$, then
   it would also prove  ${\sf con}({\sf Q})$ in $\mc J$ by the downwards persistence of $\Pi^0_1$-sentences, and, hence, 
   it would interpret ${\sf S}^1_2+{\sf con}({\sf Q})$. This contradicts the
   Second Incompleteness Theorem (our Theorem~\ref{gobu}). In fact, we can derive a direct strengthening of 
    Theorem~\ref{gobu}), to wit Theorem~\ref{pudlak2} below.
   
   \begin{theorem}[Pavel Pudl\'ak]\label{pudlak}
   The theory {\sf Q} interprets ${\sf S}^1_2$ on a definable cut. 
   \end{theorem}
   
   Pudl\'ak's result can be found in  \cite{pudl:cuts85}. The construction of the cut uses the technique of
   shortening cuts invented by Robert Solovay. A lot of the needed details were worked out by Edward Nelson in
   \cite{nels:pred86}. The locus classicus to find the detailed argument is the textbook
    \cite[Chapter 5c, pp 368-371]{haje:meta91}. However, it seems to me that the argument given there is
    incomplete. There is an almost gap-free presentation of the argument in Rachel Sterken's masters thesis
    \cite{sterk:conc08}. The corrections to  \cite{pudl:cuts85} on 
    \begin{center}
    \url{https://users.math.cas.cz/~pudlak/hpcorr.pdf}
    \end{center}
    refer the reader to my version of the proof which is quite close to what can be found in Rachel Sterken's masters thesis.
    See 
    \begin{center}
    \url{https://users.math.cas.cz/~pudlak/Q-note2.pdf}.
    \end{center}
    
    Consider any c.e. theory $U$. From Theorem~\ref{pudlak}, we immediately have that
    ${\sf Q}+{\sf con}(U)$ interprets  ${\sf S}^1_2+{\sf con}(U)$, by the downward persistence of $\Pi^0_1$-sentences.
    Thus, if 
     $U$ interprets ${\sf Q}+{\sf con}(U)$, then $U$ interprets  ${\sf S}^1_2+{\sf con}(U)$, and, hence,
   $U$ is inconsistent. Thus, we have:
   
   \begin{theorem}[Pavel Pudl\'ak]\label{pudlak2}
  Let $U$ be a c.e. theory and suppose $U$ interprets ${\sf Q}+{\sf con}(U)$. Then, $U$ is inconsistent.
   \end{theorem}
 
    If we are just after the result that {\sf Q} does not prove its own consistency, we can avoid the appeal to Theorem~\ref{pudlak} and
    make do with the much easier insight that {\sf Q} interprets ${\sf S}^1_2$. However, this has the disadvantage, that the argument
    does not deliver Theorem~\ref{pudlak2} --- as far as we know.

   When Bezboruah and Shepherdson wrote their paper around 1976, the technology needed for the Pudl\'ak result was not yet
   available. They came up with a very different technique. Their paper is called `G\"odel's Second Incompleteness Theorem for
   {\sf Q}', undoubtedly because this was the original question they started from. However, it is really about the theory ${\sf PA}^-$,
   the theory of the non-negative part of a discretely ordered commutative ring. Also, for  ${\sf PA}^-$, they show something far stronger:
   ${\sf PA}^-$ fails to prove the consistency of an extremely weak theory that I call \bs. 
   It follows that no sub-theory of ${\sf PA}^-$ proves the consistency of \bs. If this sub-theory extends \bs, then,
   \emph{a fortiori} is does not prove its own consistency.
 
   We define the theory  \bs\  as follows.\footnote{\bs\ is just an artificial theory created specifically to give the strongest possible formulation of the Bezboruah-Shepherdson result.} 
   Its language consists of two sentences $\bot$ and $(\bot \to \bot)$.
   There is one axiom, to wit $(\bot \to \bot)$ and one inference rule, Modus Ponens, to wit: from $\bot$ and $(\bot \to \bot)$ one may infer $\bot$.
   We work with a Hilbert-style proof-system. A proof is a sequence of sentences such that an occurrence of a formula is either an axiom or there
   are two strictly earlier occurrences such that the present occurrence follows from them by Modus Ponens.
      The Bezboruah-Shepherdson Theorem, in a version that does their development justice, runs as follows.
   
   \begin{theorem}\label{bezshe}
   Suppose we use the $\upbeta$-function to code sequences and that we use the sequences to code Hilbert-style proofs. Then, the theory ${\sf PA}^-$ does not prove the
   consistency of  \bs. It follows that no sub-theory of ${\sf PA}^-$  proves the consistency of
   any extension of  \bs. So, specifically, {\sf R} does not prove ${\sf con}({\sf R})$ and {\sf Q} does not prove  ${\sf con}({\sf Q})$
  \textup (under the given assumptions on arithmetisation\textup).
   \end{theorem}
   
   The idea of the proof is to create a model of ${\sf PA}^-$ with a $\upbeta$-proof that starts with
   an $\upomega$-type sequence of occurrences of $(\bot\to \bot)$ and ends with an $\upomega^\ast$-type  sequence 
   of occurrences of $\bot$, where $\upomega^\ast$ is the reverse ordering of $\upomega$. This sequence fulfils the desiderata for a Hilbert-style proof of falsum
   in  \bs.
   
   Emil Je{\v{r}}{\'{a}}bek pointed out to me that we can use a variant of \bs\ using the Genzen sequent calculus as follows.
   In stead of the axiom $(\bot\to \bot)$ we use the sequent $\bot \To \bot$ and instead of $\bot$ we use the sequent $\To \bot$. In  place of
   modus ponens,
   we have the cut rule.
   
   \begin{ques}
   {\small The model constructed by Bezboruah and Shepherdson undoubtedly satisfies many other natural true sentences.
   This creates the opportunity to read off a better result. Can you find such sentences? A specific related question is: does
   the model constructed by Bezboruah and Shepherdson satisfy all true universal sentences?}\hspace{0.05cm}\footnote{Both 
   in the Bezboruah-Shepherdson model and in the model $\mc K$ constructed in our Section~\ref{markovsmurf}, we can define the cut of the standard numbers
   in such a way that the definitions pick out the identical cut in the standard model.
   This means that we can truly add all true arithmetical sentences \emph{relativised to these 
   respective cuts}. However, I do not feel these sentences are all that natural.}
   \end{ques}
   
   We have seen that the unprovability of the consistency of {\sf Q} in {\sf Q} is a very specific consequence both of the G\"odel-Pudl\'ak Theorem~\ref{pudlak}, which is an
   extension of the Second Incompleteness  Theorem \ref{gobu}, and  of
   the Bezboruah-Shepherdson Theorem~\ref{bezshe}. Clearly, it is fair to say that these, also when considered in this specific case,  are very different results even if they coincide in content.
   
   Let us briefly review the differences between the results as stated generally.
   \begin{itemize}
   \item
   The results have very different scope. For example, there is no hope to derive the fact that {\sf R} does not prove its own consistency using the
   methods of the proof of the Pudlak Theorem~\ref{pudlak2} (since under certain conditions {\sf R} interprets ${\sf R}+{\sf con}({\sf R})$). 
   Equally, there is no hope to show that {\sf EA} does not prove its own consistency using the methods
   of the proof of the Bezboruah-Shepherdson Theorem~\ref{bezshe} (since the internal proof they construct cannot exist in a model
   of {\sf EA}).
   \item
   The precise conditions on the proof-system and the arithmetisation differ. 
   
   In the Bezboruah-Shepherdson case, we have to use the sequence coding
   provided by the $\upbeta$-function. However, it seem clear that their basic idea can be extended to other forms of sequence coding. In fact, in their
   paper Bezboruah-Shepherdson sketch an argument for their result for a second form of sequence coding. The present paper provides a third example
   in Section~\ref{markovsmurf}.
   The constraint on the arithmetisation of syntax in the Bezboruah-Shepherdson approach is zero since we only need the two concrete sentences in the
   inconsistency proof. Finally, we need to use a Hilbert-style proof-system. By \jer's observation, the argument also works for a Genzen-style system.
    Since a lot of  variation on the construction is possible, it seems plausible that using other proof systems
    would also be
   feasible. However, this would require some further work.
   
   The G\"odel-Pudl\'ak proof is not bound to a specific sequence coding. The argument works for a wide array of G\"odel numberings,
   however, there is no systematic clarity on the possible variation on G\"odel numerings of the syntax.
   For example, the function tracking negation in Sol Feferman's great paper \cite{fefe:arit60} is $\lambda x\, .\, 2\cdot 3^x$. Thus, iterating negation
   grows superexponentially.
   Does the G\"odel numbering used by Feferman work for the case of Pudl\'ak's result?  
   The G\"odel-Pudl\'ak proof  can undoubtedly be realised for all reasonable proof systems.
   \item
   The Bezboruah-Shepherdson result does not generalise to an interpretation version. For example, the result tells us that
   ${\sf R} \nvdash {\sf con}({\sf R})$. However, if we assume that the bounded quantifiers in ${\sf con}({\sf R})$ are bounded by
  variables, then it is easy to see that {\sf R} interprets ${\sf R}+{\sf con}({\sf R})$. (The result is immediate from the main result
  of \cite{viss:whyR14}. However, it can also be done `by hand'.) I do not know what happens if we drop the condition on
  the form of the bounds.
  \item
  The G\"odel-Pudl\'ak proof also shows closure under L\"ob's rule. The Bez\-bo\-ruah-Shepherdson result does not.
  We expand on this idea in Appendix~\ref{loru}.
   \end{itemize}
   
   We have seen that in their generality the G\"odel-Pudl\'ak Theorem and the Bez\-bo\-ruah-Shepherdson Theorem are very different results.
   But what about the specific consequence that {\sf Q} does not prove ${\sf con}({\sf Q})$? Are there two homophonic such theorems?
   On the one hand, there is a tendency to say \emph{no}, of course, they are one and the same theorem. Theorems are not individuated by their proofs
   (whatever Wittgenstein may have said).
   For example, the Prime Number Theorem has many proofs but there is just one such theorem. On the other hand, we have the following two considerations.
   The precise conditions on the G\"odel numbering for both applications differ. Also, the proofs lead to very different generalisations.
   In contrast the different proofs of the Prime Number Theorem do not yield very different generalisations. My inclination is to say they
   are different results, but awareness of the divergent considerations seems to be more important than the specific choice one makes. 
   
   \begin{rem}{\small
   Emil  Je{\v{r}}{\'{a}}bek, in his paper \cite{jera:sequ12}, proves that ${\sf PA}^-$ is sequential using the $\upbeta$-function. How does Je{\v{r}}{\'{a}}bek's
   work relate to the model of Bezboruah and Shepherdson?\footnote{I thank Mateusz {\L}e{\l}yk for asking me this question.} 
    \jer\ showed me an argument that the inconsistence proof in the model provided by  Bezboruah and Shepherdson and also the inconsistency
    proof in the model provided our Section~\ref{markovsmurf} 
   do not have their length in the cut that is provided in his paper. In fact, in both models \jer's cut is the standard cut. See Appendix~\ref{emilsmurf}. 
   \jer's observation still leaves open whether we can adjoin elements to the  Bezboruah {\&} Shepherdson proof, in the sense that we
   can find a sequence that has up to the length, say $\ell$, of the Bezboruah {\&} Shepherdson proof the same projections as the Bezboruah {\&} Shepherdson proof, but that
   has at $\ell$ some other projection.}\end{rem}
   
   \begin{ques}
   {\small In the Bezboruah {\&} Shepherdson model can we adjoin any element to the \bs-inconsistency proof? For example, can we adjoin $c$?}
   \end{ques} 
     
   \section{The Bezboruah-Shepherdson Theorem for Markov Coding}\label{markovsmurf}
   
In this section, I prove my version of the Bezboruah-Shepherdson Theorem.
In Section~\ref{thecoding}, I will explain the Markov coding of sequences and prove some preliminary facts
that will be used to prove the main theorem. In Section~\ref{mainsmurf}, I will give the proof.
The main theorem runs as follows.

\begin{theorem}
Let the G\"odel numbers of $\bot$ and $(\bot\to \bot)$ be positive integers.
Suppose we use the Markov coding to code sequences and that we use the sequences to code Hilbert-style proofs. Then, the theory ${\sf PA}^-$ plus all true universal sentences does not prove the
   consistency of  \bs.
\end{theorem}

We note that ${\sf PA}^-$ plus all true universal sentences is alternatively axiomatised by just the subtraction axiom 
$\vdash x\leq y \to \exists z\,\; x+z =y$
plus all true universal sentences.

The idea of the proof is as follows. We consider non-standard models of true arithmetic $\mc M$ and $\mc N$.
Here $\mc N$ is an elementary end-extension of $\mc M$. Let $a$ be in $\mc M \setminus \mathbb N$ and let
$b$ be in $\mc N \setminus \mc M$. Let $\alpha_0$ be the Markov sequence in $\mc N$ consisting of  $a$ times $(\bot\to \bot)$.
We note that $\alpha_0$ is also in $\mc M$. Let $\alpha_1$ be the Markov sequence in $\mc N$ consisting of $b$ times $\bot$.
We take $\alpha := \alpha_0\alpha_1$.
The sequence $\alpha$ is not  a proof in $\mc N$. We proceed to  thin out $\mc N$ to a model $\mc K$ of
 ${\sf PA}^-$ plus the true universal sentences, such that $\alpha$ is in $\mc K$ but $\alpha_0$ is not.
Thus, in $\mc K$, the sequence $\alpha$ will be a
  proof of $\bot$ in $\mc K$, since it misses the point where we switch from the axioms to the conclusions. 
   
   \subsection{Markov Coding}\label{thecoding}
   
   The group ${\sf SL}_2(\mathbb Z)$ consists of the $2\times 2$ matrices of integers with determinant 0 and as operation matrix
   multiplication.
   We call its  part with non-negative entries ${\sf SL}_2(\mathbb Z^+)$ or ${\sf SL}_2(\mathbb N)$.
   The  insight on which Markov-style coding is based is the fact that
${\sf SL}_2(\mathbb N)$ is isomorphic with the free monoid on two generators.
The generators are
${\tt A} := \sspl 1101$ and $ {\tt B} :=\sspl 1011$. The empty string is the identity matrix.
This basic result  is due to
Jakob Nielsen. See \cite{niel:isom18}.
The use of this insight in the context of metamathematics goes back to Andrej Markov jr in his book
\cite{mark:theo54}. It follows that ${\sf SL}_2(\mathbb N)$
satisfies the axioms of full string theory with atoms {\tt A} and {\tt B}. 

The tally-numerals ${\tt A}^n$ have the simple form ${\tt A}^n = \sspl 1n01$. 
Thus, we can code the sequence $n_0,\dots,n_{k-1}$ as ${\tt B}{\tt A}^{n_0}\dots {\tt B}{\tt A}^{n_{k-1}}$.
It is my preference to call something \emph{a sequence} if it is given with an explicit projection function.
The data structure we are looking at does not have that. It is more something like a multiset with ordered occurrences.
We call such structures \emph{container strings}.\footnote{In an earlier version of this preprint, I called these \emph{container strings}.} 
We note that container strings are perfectly adequate for
describing proofs in Hilbert-style systems since there only the order of the sub-conclusions is
relevant and never their precise location.  

We will take over this coding idea in the context of  ${\sf SL}_2(\mc R^+)$, where $\mc R$ is a discretely ordered
commutative ring and $\mc R^+$ its non-negative part.
The structure ${\sf SL}_2(\mc R^+)$
preserves sufficiently many good properties of the free monoid on two generators. Apart from the
obvious properties, we have Tarski's Editors Property:
\begin{itemize}
\item
If $\alpha \beta = \gamma\delta$, then there is an $\eta$ such that\\
\hspace*{2cm} ($\alpha = \gamma\eta$ and $\eta\beta = \delta$) or 
($\alpha\eta=\gamma$ and $\beta = \eta \delta$).  
\end{itemize}
See \cite{viss:numb25} for a proof.

We briefly look at the complexity of the definitions of some concepts in ${\sf SL}_2(\mc R^+)$.
We define the formula class ${\sf E}_n$ and ${\sf U}_n$ as follows. The classes ${\sf E_0}$
and ${\sf U}_0$ consist of the quantifier-free formulas. 
We have (with the usual stipulation that the bounding term does not contain the variable it bounds):
\begin{itemize}
\item
$\eta_{n+1} ::= \eta_n \mid \exists x\bleq t \,\upsilon_n$, where the $\eta_n$ range over ${\sf E}_n$ and the $\upsilon_n$ over ${\sf U}_n$. 
\item
$\upsilon_{n+1} ::= \upsilon_n \mid \forall x\bleq t \,\eta_n$, where the $\eta_n$ range over ${\sf E}_n$ and the $\upsilon_n$ over ${\sf U}_n$. 
\end{itemize}

The variables $\alpha,\beta,\dots$ will range over ${\sf SL}_2(\mc R^+)$.
As usual, we write $\alpha_{ij}$ for the components of $\alpha$, where the $i,j$ range over $0,1$. We define:
\begin{itemize}
\item 
We write $\alpha \leq \beta$ iff $\alpha_{ij}\leq \beta_{ij}$, for all $i,j<2$.
Moreover, $\alpha < \beta$ iff $\alpha \leq \beta$ and $\alpha \neq \beta$.
We note that, for $\alpha$, $\beta$, $\gamma$ in ${\sf SL}_2(\mc R^+)$, with
$\alpha\beta= \gamma$, we have $\alpha\leq \gamma$ and $\beta \leq \gamma$.
\item
$\alpha$ is \emph{an initial substring} of $\beta$, or $\alpha \preceq_{\sf i} \beta$, iff 
$\exists \gamma \leq \beta\;\, \beta =\alpha\gamma$.
\emph{Prima facie} this is ${\sf E}_1$, but we can rephrase it as a quantifier-free formula.
 We have:
 $\alpha \preceq_{\sf i} \beta$ iff $\alpha^{-1}\beta$ is in ${\sf SL}_2(\mc R^+)$. Thus,
$\alpha \preceq_{\sf i} \beta$ iff $\alpha_{11}\beta_{00}\geq \alpha_{01}\beta_{10}$, 
$\alpha_{11}\beta_{01}\geq \alpha_{01}\beta_{11}$, $\alpha_{00}\beta_{10}\geq \alpha_{10}\beta_{00}$, 
and $\alpha_{00}\beta_{11}\geq \alpha_{10}\beta_{01}$.
 
We define being an end-string, or $\preceq_{\sf e}$, similarly.
\item 
The matrix $\alpha$ is \emph{a container string}, or ${\sf cs}(\alpha)$, iff $\alpha$ is either the empty string, i.e., the identity matrix, or ${\tt B}\preceq_{\sf i}\alpha$.
Thus, ${\sf cs}(\alpha)$ has a quantifier-free definition.
\item
It is convenient to have an ordering on container strings too. We write $\alpha \preceq_{\sf i}^{\sf u} \beta$ for $\alpha \preceq_{\sf i}\beta \wedge
{\sf cs}(\alpha) \wedge
{\sf cs}(\alpha^{-1}\beta)$. Thus,  $ \preceq_{\sf i}^{\sf u} $ has a quantifier-free definition. Similarly, for $ \preceq_{\sf e}^{\sf u}$. 

We note that $\alpha \preceq^{\sf u}_{\sf i}\beta$ implies that ${\sf cs}(\beta)$.
Using the Editors Principle, one can show that,
over ${\sf PA}^-$, assuming that $\alpha$ and $\beta$ are container strings, $\preceq_{\sf e}$ and $\preceq^{\sf u}_{\sf e}$ coincide.
\item
The singleton container string $[n]$ is defined by $[n] := {\tt B}{\tt A}^n$.
\item
$\alpha $ is \emph{an occurrence of $n$ in $\beta$}, or ${\sf occ}(\alpha,n,\beta)$, iff
 $\alpha \preceq^{\sf u}_{\sf i}\beta \wedge [n] \preceq^{\sf u}_{\sf e}\alpha$.
So, {\sf occ} is quantifier-free.
\item
$\pi$ is a \emph{proof} (in  \bs), or ${\sf proof}_0(\pi)$,  iff\\
{\small
\hspace*{0.2cm} ${\sf cs}(\pi) \wedge \forall \alpha \bleq \pi \,\forall n \bleq \alpha_{01}\bigl ({\sf occ}( \alpha, n,\pi) \to  (n = \gnum{\bot\to \bot} \vee
n=\gnum{\bot} )\bigl) \;\wedge$\\
\hspace*{0.2cm}  $ \forall \alpha \bleq \pi \Bigl ( {\sf occ}(\alpha,\gnum \bot,\pi) \to
\exists \beta \bles \alpha\, \exists \gamma \bles \alpha\, \bigl({\sf occ}(\beta,\gnum \bot, \alpha) \wedge 
 {\sf occ}(\gamma,\gnum {\bot\to \bot}, \alpha)\bigr)\Bigr )  
$.}\\
We note that ${\sf proof}_0$ is ${\sf U}_2$.
\item
$\pi$ is \emph{a proof of $n$}, or ${\sf proof}(\pi,n)$,  iff ${\sf proof}_0(\pi) \wedge [n]\preceq_{\sf e} \pi$.
So {\sf proof} is ${\sf U}_2$. We note that $n \leq \pi_{01}$. (A warning: ${\sf PA}^-$ does not show that
every proof is a proof of something.)
\item
${\sf con}( \bs)$ iff $\forall \pi\, \neg\,{\sf proof}(\pi,\gnum\bot)$. So,
${\sf con}( \bs)$ is $\forall^\ast{\sf E}_2$. Here $\forall^\ast$ denotes a block of unbounded quantifiers.
\end{itemize}

We proceed with the things specifically needed for the next subsection. The discussion at this point is in the
real world. Later we will apply these insights also inside non-standard models of true arithmetic.

   We consider the singleton container string $[n] = {\tt B}{\tt A}^n =\sspl 1n1{n+1}$.
   The characteristic equation is $x^2-(n+2)x+1=0$.
So,    the eigenvalues of this matrix are:{\small
   \[ \lambda_n := \frac{(n+2) + \sqrt{n(n+4)}}{2} \text{ and } \mu_n := \frac{(n+2) - \sqrt{n(n+4)}}{2}.\]
   }
  We note that 
  $\lambda_n+\mu_n= n+2$ and $\lambda_n\mu_n=1$. 
  Moreover, except in case $n=0$, we have $0<\mu_n < 1<\lambda_n$ and
  $(n+1)^2< n(n+4) < (n+2)^2$ and, thus, $\lambda_n$ and $\mu_n$ are both irrational.
  
  We proceed to study $[n]^m$. We define the following sequence.
   \begin{itemize}
\item
 ${\mf s}_{n,0} = 0$, ${\mf s}_{n,1} = 1$, ${\mf s}_{n,m+1} = (n+2) {\mf s}_{n,m} - {\mf s}_{n,m-1}$.
\end{itemize}
Our $\mf s_{n,m}$ is Matiyasevich's $\uppsi_{n-2,m}$. See \cite{mati:diop71}.
This sequence satisfies:
 $[n]^m = {\mf s}_{n,m} [n] - {\mf s}_{n,{m-1}} {\sf id}$.  
 In other words,
\[
[n]^m = \sspl {\mf s_{n,m} - \mf s_{n,m-1}} {n \mf s_{n,m}}{\mf s_{n,m}}{(n+1)\mf s_{n,m} - \mf s_{n,m-1}}.
\]
  This is easily verified by induction. Since the determinant of $[n]^m$ is 1, it follows that 
  \[ \mf s^2_{n,m} - (n+2) \mf s_{n,m}\mf s_{n,m-1} + \mf s^2_{n,m-1} = 1. \]
  Matiyasevich proves that, conversely, any unordered pair of non-negative $x,y$ that satisfies $x^2-(n+2)xy+y^2=1$, is an unordered pair
  $\mf s_{n,m}$, $\mf s_{n,m-1}$.
  
From this point on, we suppose that $n \neq 0$.  If we bring $[n]$ in diagonal form, we find that it is $\gamma^{-1}\sspl {\lambda_n}00{\mu_n}\gamma$, for some $\gamma$.
  It follows that $[n]^m= \gamma^{-1}\sspl {\lambda^m_n}00{\mu^m_n}\gamma$. Hence,
  $\mf s_{n,m}$ is a linear combination of  $\lambda^m_n$ and $\mu^m_n$ (where the coefficients do not depend on $m$).
  Filling in the initial values, we can compute the coefficients, which gives us:
  \[ \mf s_{n,m} = \frac{\lambda_n^m-\mu_n^m}{\lambda_n -\mu_n}.\]

\noindent
So, the sequence $\mf s_{n,m}$ grows exponentially. We note that 
\[ \frac{\mf s_{n,m}}{\mf s_{n,m-1}} -\lambda_n = \frac{\lambda_n\mu_n^{m-1}}{\mf s_{n,m-1}}.\]
So, the sequence $\frac{\mf s_{n,m}}{\mf s_{n,m-1}}$ converges exponentially to $\lambda_n$.
A crude estimate gives $0<\frac{\mf s_{n,m}}{\mf s_{n,m-1}}-\lambda_n < \frac{1}{2^{m-1}}$ (for $n>0$).

\subsection{The Argument}\label{mainsmurf}
We are now ready to prove our version of the Bezboruah {\&} Shepherdson result for Markov coding.
We will construct a model $\mc K$ from non-standard models of true arithmetic $\mc M$ and $\mc N$. In $\mc M$ and $\mc N$, all the results of the previous subsection
will be preserved to the non-standard context. We can definitionally extend true arithmetic with
integers, rationals, and real algebraic numbers. We  only need to add finitely many real algebraic numbers,
which makes the construction reasonably simple. 
To keep things readable, we will not mention these extensions and
the embeddings of the original numbers explicitly, but talk as if, e.g., $\frac{1}{2}$ is in a given model of true arithmetic.
We treat the algebraic numbers in a model as given by an approximating sequence. E.g., $\lambda_n$ will be
the limit of $\frac{\mf s_{n,i}}{\mf s_{n,i-1}}$. We note that e.g. $\lambda_n$ may be different across models since the
approximating sequences differ.

Let $\mc A$ be a non-standard model of true arithmetic. We write $\mf Z(\mc A)$ is the closure of $\mc A$ under subtraction
using the standard pairs construction. We note that 
$\mf Z(\mc A)$ is bi-interpretable with $\mc A$.

Let $\mc M$ and $\mc N$ be non-standard models of true arithmetic, where $\mc N$ is an elementary  end-extension of $\mc M$. 
 We know that there is such an $\mc N$ by the MacDowell-Specker Theorem. See \cite{kaye:mode91}.\footnote{Since we have the freedom to
 choose $\mc M$, we can easily avoid the use of MacDowell-Specker.}
 Let $a$ a be non-standard element in $\mc M$ and let $b$ be in $\mc N\setminus \mc M$.
 
For the time being,  we work in $\mc N$. Let $m$ and $n$ be strictly positive standard numbers.
 We consider  the matrix $\beta :=   [m]^a[n]^b$. 
We set $\mc x_0 := {\mf s}_{m,a-1}$, $\mc x_1 := {\mf s}_{m,a}$, $\mc y_0 = {\mf s}_{n,b-1}$, $\mc y_1 = {\mf s}_{n,b}$.
Each entry of $\beta$
 is a linear combination of products of the form $\mc x_i\mc y_j$, with standard integer coefficients.
 
 We consider the submodel $\mc K_0$ of $\mf Z(\mc N)$ generated by the ring operations from
 the products  $\mc x_i\mc y_j$. The model $\mc K$ will be the non-negative part of $\mc K_0$.

 \begin{lem}
 \begin{enumerate}[i.]
 \item
 $\mc K_0$ is a discretely ordered commutative ring that satisfies all true universal sentences of $\mathbb Z$.
 \item 
 $\mc K$  is a model of ${\sf PA}^-$ that satisfies all true universal sentences of $\mathbb N$.
 \end{enumerate}
 \end{lem}
 
\begin{proof}
Claim (i) is immediate. Claim (ii) follows since the definition of the non-negative part is quantifier-free.
\end{proof}

 We will treat the $\mc x_i$ and $\mc y_j$ as if they were variables.
 An element of $\mc K_0$ can be written as a sum $\varSigma$ of monomials of the form $c\,\mc x_0^{k_0}\mc x_1^{k_1}\mc y_0^{\ell_0}\mc y_1^{\ell_1}$, where
 $k_0+k_1= \ell_0+\ell_1$ and $c$ is a non-zero standard integer, plus standard integer $d$. 
 We think of these forms modulo commutation and association of addition and multiplication. We will always assume that the sum of the two monomials
  $c_0\mc x_0^{k_0}\mc x_1^{k_1}\mc y_0^{\ell_0}\mc y_1^{\ell_1}$ and $c_1\mc x_0^{k_0}\mc x_1^{k_1}\mc y_0^{\ell_0}\mc y_1^{\ell_1}$ in $\varSigma$
  is contracted to the monomial $(c_0+c_1)\mc x_0^{k_0}\mc x_1^{k_1}\mc y_0^{\ell_0}\mc y_1^{\ell_1}$.

We call $k_i$ \emph{the $\mc x_i$-degree} and $\mc \ell_i$ \emph{the $\mc y_i$-degree} of the monomial. For a constant, we take these degrees to be zero.
The \emph{$\mc x$-degree} is the sum of the $\mc x_0$-degree and the $\mc x_1$-degree and, similarly, for \emph{the $\mc y$-degree}.
 
 The following lemma is our main result. I worked with ChatGPT to develop the proof. In Remark~\ref{chatsmurf}, I describe its contribution.
   
  \begin{lem}\label{goldensmurf}
 The element $\mc x_1$ is not in $\mc K_0$.
  \end{lem}
  
  \begin{proof}
 \emph{We work in $\mc N$.}   Since we have full arithmetical truth to work with, there are no worries about verifiability.
   We start with $\varSigma$ representing a polynomial that defines an element of $\mc K_0$ as above.
   We show that $\varSigma$ cannot have value  $x_1$.
  
  We say that $\Theta$ is a \emph{good} polynomial iff
  $\Theta$ is the sum of monomials of the form
   $c\mc x_0^{k_0}\mc x_1^{k_1}\mc y_0^{\ell_0}\mc y_1^{\ell_1}$ plus a standard integer,
   where $c$ is a non-zero standard integer and the difference  between the $\mc x$-degree and the $\mc y$-degree is even.
   We note that our starting point $\varSigma$ is good.
   
 We define a reduction system  on good polynomials where the basic step is either replacing an occurrence of $\mc x_1^2$ by $(m+2)\mc x_0\mc x_1-\mc x_0^2+1$ or 
 an occurrence of $\mc y_1^2$ by $(n+2)\mc y_0\mc y_1-\mc y_0^2+1$ and transforming  the resulting term
  to a good polynomial in the obvious way using the ring properties. Clearly, the transformation preserves the value and its result is again good.
  We can see that this reduction terminates in a not-necessarily unique normal form by considering the multiset of $\mc x_1$-degrees and the multiset of
  $\mc y_1$-degrees occurring in the sum. With each reduction step one of the two multisets decreases in the multiset ordering. 

Let $\varSigma^\ast$ be a normal form obtained by the rewriting. Each monomial in $\varSigma^\ast$ contains at most one occurrence of $\mc x_1$ and at most one
occurrence of $\mc y_1$.
 Let $k$ be the highest $\mc y$-degree of a monomial in $\Sigma^\ast$. We distinguish two cases: $k>0$ and $k=0$. 

Suppose $k>0$.
We consider the sum of the terms with $\mc y$-degree $k$.
This is of the form $p_0\mc y_0^k +p_1\mc y_0^{k-1}\mc y_1 = \mc y_0^k(p_0+p_1\frac{\mc y_1}{\mc y_0})$, where the $p_i$ are polynomials in the $\mc x_i$ (with standard integer
coefficients), i.e.
$\mc M$-integers, and one of the $p_i$ is non-zero.

\emph{Now for a moment we work in $\mc M$.} We note that, for some $\mc M$-rational $q$, we have $|p_0+p_1\lambda_n| > q$. 
We remind the reader that $\lambda_n$ is here  the $\lambda_n$ of $\mc M$.
Since the $\frac{\mf s_{n,i}}{\mf s_{n,i-1}}$ approximate
$\lambda_n$, it follows that, for some sufficiently large non-negative $\mc M$-integer $e$, we have: (\dag) for all $i>e$, $|p_0+p_1\frac{\mf s_{n,i}}{\mf s_{n,i-1}}| > q$.
Since $\mc N$ is an elementary extension of $\mc M$ the statement (\dag) is inherited by $\mc N$. 

\emph{We switch back to $\mc N$.}
 It follows that $|p_0+p_1\frac{\mc y_1}{\mc y_{0}}| > q$.
Hence,  $\mc y_0^k(p_0+p_1\frac{\mc y_1}{\mc y_0})$ will dominate all monomials of lower degrees and, thus, the value of $\varSigma^\ast$ cannot be $\mc x_1$.

Suppose $k=0$. In this case, $\varSigma^\ast$ is a sum of monomials of the form  $c\mc x_0^i \mc x_1$ or $c\mc x_0^j$, for $c$ a non-zero standard integer, plus a standard integer.
Here $i+1$ and $j$ are even. Suppose the largest $\mc x$-degree of a monomial in  $\varSigma^\ast$ is $\ell >0$. This implies that $\ell \geq 2$.
The sum of monomials of $\mc x$-degree $\ell$ is of the form $r_0\mc x_0^\ell +r_1\mc x_0^{\ell-1}\mc x_1 = \mc x_0^\ell(r_0+r_1\frac{\mc x_1}{\mc x_0})$, where the $r_i$ are standard integers.
In the same way as above (using $\lambda_m$ in stead of $\lambda_n$ and standard rational in stead of
$\mc M$-rational), we show that there is a standard rational $q$ such that $|r_0+r_1\frac{\mc x_1}{\mc x_0}|>q$.
Thus, $ \mc x_0^\ell(r_0+r_1\frac{\mc x_1}{\mc x_0})$ will dominate all mononomials of lower degree. Since $\ell\geq 2$, it follows that $\mc x_1$ will not be a value.

In case the highest degree $k$ is zero, the value is a standard integer, so again not equal to $\mc x_1$.
  \end{proof}
  
  The matrix $\beta =   [m]^a[n]^b$ is, almost by definition, in $\mc K$. 

\begin{lem}\label{slotsmurf}
The elements occurring in the container string $\beta$ in $\mc K$ are precisely $m$ and $n$. The first occurrence of an element is an occurrence of $m$ and the last
occurrence is an occurrence of $n$.
Moreover, every occurrence of $n$ is preceded by an occurrence of
$m$ and an occurrence of $n$.
\end{lem}

\begin{proof}
Clearly, if some elements $\vv d$ of $\mc K$ have a quantifier-free property in $\mc N$, they also have this property in $\mc K$.

Suppose $c$ in $\mc K$ is an element of $\beta$. This means that there are container strings $\gamma$ and $\delta$ in $\mc K$ such that
$\gamma[c]\delta = \beta$. Since being a container string is quantifier-free, it follows that $c$ is also an element
of $\beta$ in $\mc N$. But then $c=m$ or $c = n$. 

We easily see that $\beta_0 :=[m]^{-1}\beta$ is in $\mc K_0$, since its entries are linear combinations with coefficients if $\mathbb Z$ of the
entries of $\beta$.
Since $\beta_0$ is a container string in $\mc N$, it is a container string in $\mc K$. 
Similarly, we see that $n$ is the last element that occurs in $\beta$ by considering $\beta_1:= \beta[n]^{-1}$.

Consider any occurrence of $n$ in $\beta$ according to $\mc K$. Suppose we have, in $\mc K$, that $\gamma[n]\delta = \beta$,
for container strings $\gamma$ and $\delta$. We consider $\gamma_0 :=[m]^{-1}\gamma[n]^{-1}$. It is easily seen that
$\gamma_0$ is in $\mc K_0$, since its entries are linear combinations with coeffients in $\mathbb Z$ of the entries of $\gamma$.
We note that, in $\mc N$, the matrix $\gamma$ is not $[m]^a$, since $\mc x_1$ is an entry of 
$[m]^a$.
Thus,  $\gamma$ will end in $n$ and, hence, $\gamma_0$ will be a container string in $\mc N$ and, thus, also a container string in $\mc K$. So, we see that occurrences of $m$ and
$n$ precede the given occurrence of $n$.
\end{proof}

With Lemma~\ref{slotsmurf}, we are done. Taking for $m$ the G\"odel number of $(\bot\to\bot)$ and for $n$ the G\"odel number of $\bot$,
we find a Hilbert-style proof of $\bot$. We have to assume that the  G\"odel numbers of $(\bot\to\bot)$ and $\bot$ are non-zero.
We note that, for our applications to stronger systems, nothing depends on the
presence of $\bot$ in the language. We can replace $\bot$ by any sentence $\phi$ in the language of the theory at hand.

\begin{rem}
{\small  The means used to prove our result could be lighter. In the first place we do not really need $\mc N$ to be an elementary \emph{end-}extension of $\mc M$.
An elementary extension simpliciter would suffice. In the second place, we do not need the full power of being elementary, $\Sigma^0_n$-elementary, for sufficiently
large $n$, suffices. This last choice would have the advantage that we can make the extension $\mc N$ definable in $\mc M$. This would allow us
to have $\mc M$ think about $\mc N$, which could improve the presentation heuristically.

It seems probable to me that, instead of having a second non-standard model, we can choose $b$ large enough in $\mc M$, perhaps $b=2^a$. However, I did not
explore this possibility. Similarly, we could try to avoid non-standard models of a substantial portion of arithmetic and work with a suitable
quotient of $\mathbb Z[{\sf X},{\sf Y}]$. It is somewhat subtle to define the right ordering here. 
It would be an added insight if we could make such a model decidable. Again, I have not explored this possibility.}
\end{rem}

For our present purposes, we do not need more information about our model $\mc K$. However it is always nice to know a bit more.

\begin{theorem}\label{bythewaysmurf}
\begin{enumerate}[i.]
\item
The element $\mc x_0^2$ is in $\mc K$. 
\item
For every $u$ in $\mc K$, there is a standard positive  $p$ such that $\frac {\mc x_0^2} p < u$.
\item
For every standard positive $r$, there is a $v$ in $\mc K$ such that $v  <  \frac {\mc x_0^2} r $.
\end{enumerate}
 \end{theorem}

\noindent
In other words, the non-standard elements of $\mc K$ are co-initial in the set 
\[\{u\mathop{\in} \mc K \mid \exists p \mathop{\in} \omega{\setminus} \{0\}\;\, \frac{\mc x_0^2}p \leq u\}.\]

\begin{proof}
Let $m:= \gnum {(\bot \to \bot)}$ and $n := \gnum \bot$.
In $\mc N$, we have \[\mc x_0^2 = \mc x_0^2\mc y_1^2 - (n+2) \mc x_0^2\mc y_0 \mc y_1+\mc x_0^2\mc y_0^2.\] So, $\mc x_0^2$ is in $\mc K$.

The elements of  $\mc x$-degree 2 and $\mc y$-degree 0 are co-initial in the non-standard elements of $\mc K$.
So, it suffices to show the theorem for the elements  of  $\mc x$-degree 2 and $\mc y$-degree 0.
Inspecting the proof of Lemma~\ref{goldensmurf}, we see that these can be written in the form
$k_2\mc x_0^2+k_1\mc x_0\mc x_1+k_0$, where the $k_i$ are standards integers.
We rewrite these elements to the form $\mc x_0^2(k_2+k_1\frac{\mc x_1}{\mc x_0})+k_0$.

We remind the reader that $0<\frac{\mc x_1}{\mc x_0} -\lambda_m < \frac{1}{2^{a-1}} $.
We note that, if $k_2+ k_1\lambda_m$ is positive, then $k_2+ k_1\lambda_m > \frac 1p$, for some standard $p$, since $\lambda_m$ is irrational. This clearly remains the
same if we replace $\lambda_m$ by $\frac{\mc x_1}{\mc x_0}$.

By Dirichlet's Approximation Theorem,\footnote{See, e.g., the Wikipedia Lemma on this theorem} for every $r>0$, we  can find $k_2$ and $k_1$ such that
 $0<k_2+ k_1\lambda_m < \frac 1r$. Again this is preserved 
 if we replace $\lambda_m$ by $\frac{\mc x_1}{\mc x_0}$.
\end{proof}

\noindent
We have the following immediate corollary.

\begin{cor}
The ordering of the elements of $\mc K$ modulo finite differences is dense.
\end{cor}

\noindent Clearly, if $\mc K$ is countable, the ordering modulo finite differences has the order type of  $1+\mathbb Q$.
We note that the Bezboruah-Shepherdson model has, in the ordering modulo finite
differences, an immediate successor of the bottom. 

\begin{rem}\label{chatsmurf}
{\small I worked with ChatGPT on the proof Lemma~\ref{goldensmurf}. The program provided the recursion equations for the ${\mf s}_{n,k}$ and their characterisation
in terms of the eigenvalues. I am aware that this is standard material, but, since my knowledge of matrix theory was somewhat rusty, it greatly sped up the process.
The truly original contribution of ChatGPT is that it found a precursor of \textup{(\dag)} in the argument.

For the proof of Theorem~\ref{bythewaysmurf}, ChatGPT directed me to Dirichlet's Approximation Theorem.}
\end{rem}


\appendix

\section{Avoiding the Cut-interpretation of ${\sf S}^1_2$ in {\sf Q}}\label{gapsmurf}
   
    If we are just interested in ${\sf Q}\nvdash {\sf con}({\sf Q})$, we can side-step the need for Theorem~\ref{pudlak} as follows.
    By Nelson's work, we already know that {\sf Q} interprets ${\sf S}^1_2$. 
    Suppose ${\sf Q}\vdash {\sf con}({\sf Q})$. Then,
    ${\sf S}^1_2\vdash {\sf con}({\sf Q})$. So, {\sf Q} interprets  ${\sf S}^1_2+ {\sf con}({\sf Q})$. \emph{Quod non} by Theorem~\ref{gobu}.
  
    The disadvantage of this cut-evading argument is that it does not give us the \emph{essential} character of the Second Incompleteness Theorem.
   In other words, we do not get Theorem~\ref{pudlak2} (as far as we know). As a special case, the evasive argument does not
    help to prove closure of {\sf Q} under L\"ob's Rule. In contrast,  the Pudl\'ak argument using a cut does deliver the desired goods.

\section{Provability Logic for ${\sf PA}^-$} \label{loru}
Consider ${\sf PA}^-$. We use a sequence coding that is closed under concatenation of sequences or container strings. We can
obtain such a coding from Je{\v{r}}{\'{a}}bek's result in \cite{jera:sequ12} by shortening the length of the sequences
further or by using the Markov style container strings. We use a Hilbert-style system or sequent system.
Under these stipulations, we get the second L\"ob Condition:  $\opr(\phi \to \psi)$ implies $(\opr \phi \to \opr\psi)$.
Generally, we get Necessitation and, by the G\"odel-Pudl\'ak argument, we get L\"ob's rule. Moreover, we have modalised fixed points.
This gives us Cyclic Henkin Logic {\sf CHL}. This logic is studied in \cite{viss:cycl21}. A salient sub-system of this logic is Well-foundedness Logic
  {\sf WfL}. We give this logic below. The main idea here is to define a new operator $\opr^\grullet$ given 
  by the modalised fixed point equation $\opr^\grullet\phi \iff \opr(\phi\wedge \opr^\grullet \phi)$. The new operator
  will satify L\"ob's Logic.

Well-foundedness Logic {\sf WfL} has the modal language with two unary modalities $\opr$ and
$\opr^\grullet$. It is axiomatised by all propositional tautologies, closure under modus ponens, plus the following principles
and rules:
\begin{enumerate}[{\sf wfl}1.]
\item
If $\vdash \phi$, then $\vdash \opr\phi$.
\item
If $\vdash \opr \phi \to \phi$, then $\vdash \phi$.
\item
$\vdash \opr(\phi \to \psi) \to (\opr \phi \to \opr\psi)$.
\item
$\vdash \opr^\grullet\phi \iff \opr(\phi\wedge \opr^\grullet \phi)$.
\end{enumerate}

From these axioms and rules we can derive that $\opr^\grullet$ satisfies L\"ob's Logic.
Also, the  de Jongh-Sambin-Bernardi Theorem  on uniqueness of modalised fixed points follows. See \cite{viss:cycl21} for details.

We can easily see that {\sf WfL} is sound for interpretations in ${\sf PA}^-$ when we interpret $\opr$ as $\opr_{{\sf PA}^-}$ 
and $\opr^\grullet$ as described above.

We have two conjectures.

\begin{conj}\label{sterkesmurf}
{\small Suppose we use Markov-style container strings to represent proofs. Then, {\sf WfL} is the provability logic of ${\sf PA}^-$ for the $\opr,\opr^\grullet$-language.}
\end{conj}

\begin{conj}\label{spiersmurf}
{\small Suppose we use Markov-style container strings to represent proofs. Then, {\sf WfL} is closed under the admissible rule
$\opr p \to \opr q\;/\; p\to q$, in other words, for any arithmetical $\phi$ and $\psi$, if ${\sf PA}^- \vdash \opr_{{\sf PA}^-} \phi \to \opr_{{\sf PA}^-}\psi$,
then ${\sf PA}^- \vdash  \phi \to \psi$}
\end{conj}

We note that the truth of Conjecture~\ref{sterkesmurf} implies closure under $\opr p \to \opr q\;/ \; p\to q$ for provability logic of ${\sf PA}^-$, but does not imply
the truth of Conjecture~\ref{spiersmurf}.

A reasonable first question to look at would be the following.

\begin{ques}
{\small Suppose we use Markov-style container strings to represent proofs. We adapt the model $\mc K$ from Section~\ref{markovsmurf} to a model $\mc K^\ast$ by 
replacing $\bot$ by the G\"odel sentence $\gamma$ for $\opr_{{\sf PA}^-}$. Do we have an internal proof of $\bot$ in $\mc K^\ast$?}
\end{ques}

We note that in the study of this system both the G\"odel-Pudl\'ak argument and the methods initiated by Bezboruah {\&} Shepherdson could play a role.

\section{An Argument by Emil \jer}\label{emilsmurf}

Consider a commutative ring $\mc R$. Two elements $a$ and $b$ are called \emph{co-prime} or \emph{co-maximal} if
they satisfy the B\'ezout identity $xa+yb=1$, for some $x$ and $y$. We will use the following elementary facts that 
hold in any commutative ring.
\begin{enumerate}[a.]
\item
Suppose $j \neq i$,  $a= 1+iv$, $b= 1+jv$ and $(j-i) \mid v$.
Then $a$ and $b$ are co-prime.
\item
Suppose $a$ and $c$ are co-prime and $b$ and $c$ are co-prime. Then, $ab$ and $c$ are co-prime. 
\item
 Suppose $a$ and $b$ are co-prime and $a\mid c$ and $b\mid c$. Then, $ab\mid c$.
\end{enumerate}

\medskip
Consider any model $\mc M$ of ${\sf PA}^-$ that contains an element $c$ such that the sequence $(c^n)_{n\in \omega}$ is cofinal. 
The model constructed by Bezboruah {\&} Shepherdson satisfies this condition on $\mc M$.
Let $\mc J$ be the class of all
$a$ in $\mc M$ such that \[ (\dag) \;\;\;\forall v\, \exists u\,\forall i \bleq a\,\; 1+iv \mid u.\] Clearly, $\mc J$ is a cut. The cuts $I_1$ and $I_2$ of \jer's paper are 
sub-cuts of $\mc J$. We show that $\mc J$ is the standard cut.

Suppose $a > \omega$ is in $\mc J$. We apply (\dag) for $v:=1$. We get $u_0$ such that 
$\forall i \bleq a\, 1+i \mid u_0$.  We apply (\dag) for $v := u_0c$. This delivers $u$.
Suppose $i<j \leq a$. Then, $j-i\leq a$. So, $j-i\mid u_0c$. Applying (a) above, this tells us
that $1+iu_0c$ and $1+ju_0c$ are co-prime. 

We prove by induction that, for every standard
$n$, we have  \[(1+u_0c) \dots (1+nu_0c) \mid u.\] This certainly holds for $0$.
Suppose we have  $(1+u_0c) \dots (1+ku_0c) \mid u$. Since, each of $1+iu_0c$ for $i\leq k$ is
co-prime with  $(1+(k+1)u_0c)$, we find, by repeated application of (b), that $(1+u_0c) \dots (1+ku_0c)$
is co-prime with $(1+(k+1)u_0c)$. So, by application of (c), we find that 
$(1+u_0c) \dots (1+ku_0c)(1+(k+1)u_0c) \mid u$.

We may conclude that  
 $c^n < u$, for all $n\in \omega$. But this is impossible.
 \end{document}